\documentclass[12pt,leqno,oneside]{amsart}
\usepackage{mathrsfs,dsfont}
\usepackage{amsmath,amstext,amsthm,amssymb,amscd}
\usepackage{charter}
\usepackage{typearea}
\usepackage{pdfsync}
\usepackage{mathtools}
\usepackage[width=6.4in,height=8.5in]{geometry}

\numberwithin{equation}{section}

\newtheorem{Theorem}{Theorem}[section]
\newtheorem{Proposition}[Theorem]{Proposition}
\newtheorem{Lemma}[Theorem]{Lemma}
\newtheorem{Corollary}[Theorem]{Corollary}
\theoremstyle{definition}
\newtheorem{Definition}[Theorem]{Definition}

\newtheorem*{problem}{Problem} 
\newtheorem*{problem+}{Problem+} 
\newtheorem*{ackn}{Acknowledgements}

\newtheorem{Example}{Example}

\newcommand{\ddbar}{{\partial\overline\partial}}
\newcommand{\db}{\overline\partial}

\DeclareMathOperator{\Pic}{Pic}
\newcommand{\cali}[1]{\mathscr{#1}}
\newcommand{\cO}{\cali{O}}

\newcommand{\cC}{\cali{C}}

\newcommand{\field}[1]{\mathbb{#1}}
\newcommand{\Z}{\field{Z}}
\newcommand{\R}{\field{R}}
\newcommand{\C}{\field{C}}
\newcommand{\N}{\field{N}}

\newcommand{\Q}{\field{Q}}

\renewcommand{\P}{\field{P}}

\newcommand{\comment}[1]{}

\begin{document}

\title{On the extension of quasiplurisubharmonic functions} 
\author{Dan Coman}
\thanks{D.\ Coman is supported by the Simons Foundation grant No.\ 853088}
\address{Department of Mathematics, Syracuse University, 
Syracuse, NY 13244-1150, USA}\email{dcoman@syr.edu}
\author{Vincent Guedj}
\thanks{V.\ Guedj is partially supported by the CIMI project Hermetic, ANR-11- LABX-0040}
\address{Institut de Math\'ematiques de Toulouse, Universit\'e Paul Sabatier, 118 route de Narbonne, 31062 Toulouse Cedex 9, FRANCE}
\email{vincent.guedj@math.univ-toulouse.fr}
\author{Ahmed Zeriahi}
\address{Institut de Math\'ematiques de Toulouse, Universit\'e Paul Sabatier, 118 route de Narbonne, 31062 Toulouse Cedex 9, FRANCE}
\email{ahmed.zeriahi@math.univ-toulouse.fr}

\date{March 11, 2022}

\pagestyle{myheadings} 

\begin{abstract}
Let $(V,\omega)$ be a compact K\"ahler manifold such that $V$ admits a cover by Zariski-open Stein sets with the property that 
$\omega$ has a strictly plurisubharmonic exhaustive potential on each element of the cover.
 If $X\subset V$ is an analytic subvariety, we prove that any $\omega|_X$-plurisubharmonic function on $X$ extends to a $\omega$-plurisubharmonic function on $V$. 
 
 This result generalizes a previous result of ours on the extension of singular metrics of ample line bundles.
 It allows one to show that any transcendental K\"ahler class in the real Neron-Severi space $NS_{\R}(V)$   has this extension property.
\end{abstract}


\maketitle

\begin{center}
{\em Dedicated to L\'aszl\'o Lempert in honor of his 70th birthday}
\end{center}

\section{Introduction}\label{S:intro}

Let $(V,\omega)$ be a compact K\"ahler manifold of dimension $n$. 
We recall that a function $\psi:V\to[-\infty,+\infty)$ is called quasiplurisubharmonic (qpsh) if it is locally the sum of a plurisubharmonic (psh) function and a smooth one.
 If $\psi$ is qpsh then 
  $\psi\in L^1(V,\omega^n)$
  and 
   $\psi$ is called $\omega$-plurisubharmonic ($\omega$-psh) if $\omega+dd^c\psi\geq0$ in the sense of currents on $V$. Here $d=\partial+\db$, $d^c=\frac{1}{2\pi i}\,(\partial-\db)$. We denote by $PSH(V,\omega)$ the class of $\omega$-psh functions on $V$ and refer the reader to \cite{GZ,GZ17} for their basic properties. 
   
   \smallskip

Let now $X\subset V$ be a (closed) analytic subvariety. An upper semicontinuous function $\varphi:X\to[-\infty,+\infty)$ is called $\omega|_X$-psh if $\varphi\not\equiv-\infty$ on $X$ and if 
there exist an open cover $\{U_\alpha\}_{\alpha\in A}$ of $X$ and psh functions $\varphi_\alpha,\rho_\alpha$ defined on $U_\alpha$, where $\rho_\alpha$ is smooth and $dd^c\rho_\alpha=\omega$, such that $\rho_\alpha+\varphi=\varphi_\alpha$ holds on $X\cap U_\alpha$, for every $\alpha\in A$. 
Note that we allow $\varphi$ to be identically $-\infty$ on some (but not all) irreducible components of $X$. 

The function $\varphi$ is called {\em strictly} $\omega|_X$-psh if $\varphi$ is not identically $-\infty$ on any irreducible component of $X$ and if it is $(1-\varepsilon)\omega|_X$-psh for some small $\varepsilon>0$. The current $\omega|_X+dd^c\varphi$ is then called a K\"ahler current on $X$. 
We denote by $PSH(X,\omega|_X)$, resp. $PSH^+(X,\omega|_X)$, the class of $\omega|_X$-psh, resp. strictly $\omega|_X$-psh functions on $X$. For a detailed discussion of psh functions on complex spaces, we refer the reader to \cite{FN} and \cite[section 1]{De85}. 

\smallskip

By restriction, $\omega$-psh functions on $V$ yield $\omega|_X$-psh functions on $X$. An interesting problem is whether every $\omega|_X$-psh function $\varphi$ on $X$ arises in this way:

\begin{problem}
Does one have
$
PSH(V,\omega)|_X=PSH \left(X,\omega|_X \right) \,  ?
$
\end{problem}

In \cite[Theorem B]{CGZ13}, this was shown to be the case when $\omega$ is a Hodge form, i.e. $\omega$ is a representative of the Chern class $c_1(L)$ of a positive holomorphic line bundle $L$ over $V$ (this requires the manifold $V$ to be projective). The proof relies on an extension result with growth control of psh functions on analytic subvarieties of Stein manifolds \cite[Theorem A]{CGZ13}. This in turn used methods of Coltoiu based on Runge domains \cite[Proposition 2]{Col91}, and of Sadullaev \cite{Sa82} (see also \cite[Theorem 3.2]{BlLe03}).

\smallskip

Besides being a natural one, this problem has found important applications in K\"ahler geometry,
allowing one to prove the continuity of K\"ahler-Einstein potentials 
(see e.g. \cite{EGZ,EGZ17,Dar17,GGZ20}).

\smallskip

A related and slightly simpler question is:

\begin{problem+}
Does one have
$
PSH^+(V,\omega)|_X=PSH^+ \left(X,\omega|_X \right) \,  ?
$
\end{problem+}

When $X$ is a complex submanifold of $V$, it is not difficult to show that any smooth strictly $\omega|_X$-psh function on $X$ extends to a smooth strictly $\omega$-psh function on $V$ (see e.g. \cite[Proposition 2.1]{CGZ13} and  \cite{Sch}). 

In \cite{CT14}
 it was shown that the answer to Problem+ is again positive for functions 
with analytic singularities. 
The proof 
uses resolution of singularities and a gluing argument similar to that of Richberg \cite{R68}. 

A further result in this direction was recently obtained in \cite{WZ20,NWZ21}. 
Assuming that there exists a holomorphic retraction $U\to X$ on a neighborhood $U\subset V$ of $X$,
 it is proved in \cite{NWZ21} that any strictly $\omega|_X$-psh function on $X$ extends to a strictly $\omega$-psh function on $V$. 
These results 
all assume that $X$ is smooth.

\smallskip

When $X$ is singular, the equality 
$PSH^+(V,\omega)|_X=PSH^+ \left(X,\omega|_X \right)$
holds if the 
class $\{\omega\}$ belongs to the real Neron-Severi space $NS_{\R}(V)$.
This follows from \cite[Theorem B]{CGZ13} through a simple density argument. Recall that
$NS_{\R}(V)=NS(V) \otimes \R$, where
$$
NS(V)=H^{1,1}(V,\R) \cap H^2(V,\Z)/ torsion.
$$
Thus Hodge classes are rational points of $NS_{\R}(V)$.

\medskip

The main goal of this note is to show the following stronger result:

\begin{Theorem}\label{T:NS}
Let $(V,\omega)$ be a compact K\"ahler manifold and let $X$ be an analytic subvariety of $V$.
If $\{\omega\} \in NS_{\R}(V)$ then
$
PSH(V,\omega)|_X=PSH \left(X,\omega|_X \right).
$
\end{Theorem}

Note that a compact K\"ahler manifold $(V,\omega)$ such that $\{\omega\} \in NS_{\R}(V)$ must be projective (see Lemma \ref{L:NS}).

\smallskip

We actually show that the methods developed in \cite{CGZ13} can be adapted to prove that this extension property
holds for a more general class of compact K\"ahler manifolds $(V,\omega)$ than the one of polarized projective manifolds 
considered in \cite[Theorem B]{CGZ13}:

 \smallskip
 
 \begin{Definition}
 We say that $(V,\omega)$ admits an {\em adapted Zariski-open Stein cover} if 
 \begin{itemize}
 \item  $V$ can be covered by finitely many open Stein sets $V_j$, $1\leq j\leq N$, such that $H_j:=V\setminus V_j$ is an analytic subvariety of $V$, and
 \item  there exist smooth exhaustion functions $\rho_j\geq0$ on each $V_j$ such that $\omega=dd^c\rho_j$.
 \end{itemize}
 In this case we also say that $(V,\omega)$ verifies condition (ZOS).
 \end{Definition}
 
 Note that the sets $V_j$ are connected, as we assume $V$ is connected and since $H_j$ are analytic subvarieties. This notion has the following basic properties:

\begin{Proposition}\label{P:cc}
Let $(V,\omega)$ be a compact K\"ahler manifold. We have the following:

(i) Condition (ZOS) only depends on the cohomology class of $\omega$.

(ii) The set of K\"ahler classes $\{\omega\}$ such that $(V,\omega)$ satisfies condition (ZOS) is a convex cone.

(iii) If $\{\omega\}$ is a Hodge class then $(V,\omega)$ satisfies condition (ZOS).
\end{Proposition}

It follows from the above observations that condition (ZOS) is satisfied for
any K\"ahler class in the real Neron-Severi space $NS_{\R}(V)$ (see Lemma \ref{L:NS}).
We show moreover that these conditions  are stable under restriction and products, as well as by pull-back by holomorphic coverings.

\smallskip

The main technical result of the note is the following extension theorem in this setting,
which is a generalization of  \cite[Theorem B and Theorem 2.2]{CGZ13}:
 
\begin{Theorem}\label{T:mt}
Let $(V,\omega)$ be a compact K\"ahler manifold of dimension $n$ which verifies condition (ZOS) and let $X$ be an analytic subvariety of $V$. If $\varphi \in PSH(X,\omega|_X)$ then, given any constant $a>0$, there exists $\psi\in PSH(V,\omega)$ such that $\psi|_X=\varphi$ and  $\max_V\psi<\max_X\varphi+a$.
\end{Theorem}

Note that in Theorem \ref{T:mt} we  do not assume that $V$ is projective. Since qpsh functions on a compact K\"ahler manifold can be regularized (see \cite{De92,BK07}), one has the following immediate corollary:

\begin{Corollary}\label{C:reg} 
Let $(V,\omega)$ be a compact K\"ahler manifold which verifies condition (ZOS) and let $X$ be an analytic subvariety of $V$.
 If $\varphi \in PSH(X,\omega|_X)$ then there exists a sequence of smooth functions 
 $\varphi_j\in PSH(X,\omega|_X) \cap {\mathcal C}^{\infty}(X)$ which decrease pointwise to $\varphi$.
\end{Corollary}

This smoothing property plays a key role in establishing global continuity of singular K\"ahler-Einstein potentials
(see \cite[Theorem 2.1]{EGZ}). 
Corollary \ref{C:reg} thus provides an alternative proof of \cite[Theorem 3.9]{GGZ20}.
Note that the regularizing techniques of Demailly \cite{De92} and B\l ocki-Kolodziej \cite{BK07} break
down when $X$ is singular.

\medskip

Theorem \ref{T:mt} is proved in Section \ref{S:mt}. Theorem \ref{T:NS} follows from Theorem \ref{T:mt} and Proposition \ref{P:cc} (see Section \ref{SS:BP} $(v)$ and Lemma \ref{L:NS}). In Section \ref{S:ex} we prove Proposition \ref{P:cc} and we collect some examples of compact K\"ahler manifolds which admit an adapted Zariski-open Stein cover. Following \cite{Mat13}, we show in Proposition \ref{pro:matsumura} that the reference cohomological class has to be K\"ahler for the extension property to hold. 

\begin{ackn} 
We thank S\'ebastien Boucksom for several useful discussions, and the referee for pointing out a gap in the original proof of Proposition \ref{pro:matsumura}.

It is a pleasure to dedicate this work to L\'aszl\'o Lempert.
His contributions to complex analysis and geometry have been both profound and very elegant.
\end{ackn}

\section{Proof of Theorem \ref{T:mt}}\label{S:mt}

Let $X$ be an analytic subvariety of a compact K\"ahler manifold $(V,\omega)$ of dimension $n$, which verifies condition (ZOS). Note that the function $-\rho_j\leq0$ is an $\omega$-psh function on $V_j$. Since $V\setminus V_j=H_j$ is an analytic set, $-\rho_j$ extends to an $\omega$-psh function $\theta_j$ on $V$, such that $\theta_j\leq0$ on $V$. Since $\rho_j$ is an exhaustion of $V_j$, we see that $\theta_j=-\infty$ on $H_j$. Let 
\begin{equation}\label{e:theta}
\theta:=\max\{\theta_j:\,1\leq j\leq N\}.
\end{equation}
It follows that $\theta\leq0$ is a continuous $\omega$-psh function on $V$, and we set 
\begin{equation}\label{e:m}
m:=-\min_{z\in V}\theta(z)>0.
\end{equation}

\par The following lemma gives special subextensions of $\omega|_X$-psh functions on $X$. 

\begin{Lemma}\label{L:subext}
Let $\varepsilon\geq0$ and $u$ be a continuous $(1+\varepsilon)\omega$-psh function on $V$ such that $u(z)\leq0$ for all $z\in V$. If $c>1$ and $\varphi$ is an $\omega|_X$-psh function on $X$ such that $\varphi<u$, then there exists a $c\omega$-psh function $\psi$ on $V$ such that 
\[\frac{1}{c}\,\psi(z)\leq\frac{1}{1+\varepsilon}\,u(z),\;\forall z\in V,\]
and 
\[\psi(z)=\varphi(z)+(c-1)\theta(z)+(c-1)\min_{\zeta\in V}u(\zeta),\;\forall z\in X.\]
\end{Lemma}

\begin{proof} Let $X_j=X\cap V_j$. We may assume that $X_j\neq\emptyset$ for $1\leq j\leq N'$, and $X_j=\emptyset$ for $N'<j\leq N$, where $N'$ is an integer with $1\leq N'\leq N$. Set 
\[M=-\min_{\zeta\in V}u(\zeta)\geq0.\]
Let $1\leq j\leq N'$. The function $\varphi+\rho_j+M$ is psh on $X_j$ and since $u\leq0$ we have that 
\[\varphi+\rho_j+M<u+\rho_j+M\leq\frac{1}{1+\varepsilon}\,u+\rho_j+M\,\text{ on } X_j.\]

Fix $c',\,1<c'<c$. Note that the function $(1+\varepsilon)^{-1}u+\rho_j+M\geq0$ is a continuous psh exhaustion function of $V_j$. It follows from \cite[Theorem A]{CGZ13} that there exists a psh function $\widetilde\psi_j$ on $V_j$ such that 
\[\widetilde\psi_j<\frac{c'}{1+\varepsilon}\,u+c'\rho_j+c'M\;{\rm on}\;V_j,\,\;\widetilde\psi_j=\varphi+\rho_j+M\;{\rm on}\;X_j.\] 
The function 
\[\psi_j:=\widetilde\psi_j-c\rho_j-cM\]
is $c\omega$-psh on $V_j$ and, since $c'<c$, it verifies 
\[\psi_j<\frac{c}{1+\varepsilon}\,u+(c'-c)\rho_j+(c'-c)\big(\frac{1}{1+\varepsilon}\,u+M\big)\leq\frac{c}{1+\varepsilon}\,u\leq0.\]
Since $\rho_j(z)\to+\infty$ as $V_j\ni z\to H_j$ and since $u$ is continuous on $V$, we conclude that $\psi_j$ extends to a $c\omega$-psh function $\psi_j$ on $V$ such that 
\[\psi_j\leq\frac{c}{1+\varepsilon}\,u\,\text{ on } V,\,\;\psi_j=-\infty\,\text{ on } H_j,\,\;\psi_j=\varphi-(c-1)\rho_j-(c-1)M\,\text{ on } X_j.\]
Hence $\psi_j=\varphi+(c-1)\theta_j-(c-1)M$ on $X$, as both sides are equal to $-\infty$ on $X\cap H_j$. 

We define $\psi:=\max_{1\leq j\leq N'}\psi_j$. Then $\psi$ is $c\omega$-psh on $V$ and 
\[\psi\leq\frac{c}{1+\varepsilon}\,u\,\text{ on } V,\,\;\psi=\varphi+(c-1)\max_{1\leq j\leq N'}\theta_j-(c-1)M\,\text{ on } X.\]
If $j>N'$ then $X\subset H_j$, so $\theta_j=-\infty$ on $X$. Therefore
\[\theta=\max_{1\leq j\leq N}\theta_j=\max_{1\leq j\leq N'}\theta_j,\,\;\psi=\varphi+(c-1)\theta-(c-1)M,\]
hold on $X$.
\end{proof}

\smallskip

\begin{proof}[Proof of Theorem \ref{T:mt}] Using Lemma \ref{L:subext}, we repeat the proof of \cite[Theorem 2.2]{CGZ13}. Assuming that $\max_X\varphi=-a$, where $a>0$, we have to find a negative $\omega$-psh function $\psi$ on $V$ with $\psi=\varphi$ on $X$. Let $X'$ be the union of the irreducible components $Y$ of $X$ such that $\varphi|_Y\not\equiv-\infty$. We construct by induction on $j\geq1$ a sequence of numbers $\varepsilon_j\searrow0$ and a sequence of negative smooth $(1+\varepsilon_j)\omega$-psh functions $\psi_j$ on $V$ such that, for all $j\geq2$, 
\[\frac{\psi_j}{1+\varepsilon_j}<\frac{\psi_{j-1}}{1+\varepsilon_{j-1}}\;\;{\rm on}\;V\;,\;\;\psi_{j-1}>\varphi\;{\rm on}\;X\;,\;\;\int_{X'}(\psi_j-\varphi)<\frac{1}{j}\;,\;\;\int_{W}\psi_j<-j\,,\]
for every irreducible component $W$ of $X$ where $\varphi|_W\equiv-\infty$. Here the integrals are with respect to the area measure on each irreducible component $X_j$ of $X$, i.e.
\[\int_Xf:=\sum_{X_j}\int_{X_j}f\,\omega^{\dim X_j}\,,\,\;|X_j|:=\int_{X_j}\omega^{\dim X_j}.\]

Let $\varepsilon_1=1$, $\psi_1=0$, and assume that $\varepsilon_{j-1},\,\psi_{j-1}$, where 
$j\geq2$, are constructed with the above properties. Since $\varphi<\psi_{j-1}|_X$ and the latter is continuous on the compact set $X$, we can find $\delta>0$ such that $\varphi<\psi_{j-1}-\delta$ on $X$. 

Let $c>1$. Lemma \ref{L:subext} implies that there exists a $c\omega$-psh function $\psi_c$ such that 
\[\frac{\psi_c}{c}\leq\frac{\psi_{j-1}-\delta}{1+\varepsilon_{j-1}}\;\;{\rm on}\;V\;,\;\;\psi_c=\varphi+(c-1)\theta-(c-1)M_{j-1}\;\;{\rm on}\;X,\]
where $M_{j-1}=\delta-\min_{\zeta\in V}\psi_{j-1}(\zeta)\geq0$. 
By \cite{BK07} and \cite{De92}, there exists a sequence of smooth $c\omega$-psh functions $\{\eta_k\}$ decreasing pointwise to $\psi_c$ on $V$. Let $\psi'_c:=\eta_k$, where $k$ is chosen sufficiently large such that $\psi'_c$ is a negative smooth $c\omega$-psh function on $V$ and it verifies the following: 
\begin{align*}
\frac{\psi'_c}{c}&<\frac{\psi_{j-1}-\frac{\delta}{2}}{1+\varepsilon_{j-1}}\;\,\text{ on $V$}, \\
\psi'_c&>\varphi+(c-1)\theta-(c-1)M_{j-1}\geq\varphi-(c-1)(m+M_{j-1})\;\,\text{ on $X$}, \\
\int_{X'}(\psi'_c-\varphi)&\leq\int_{X'}(\psi'_c-\varphi-(c-1)\theta+(c-1)M_{j-1})<c-1, \\
\int_{W}\psi'_c&<-j-(c-1)(m+M_{j-1})|W|,
\end{align*}
where $W$ is as above. The last two requirements are possible thanks to the dominated, respectively monotone, convergence theorems.

Let $\psi''_c:=\psi'_c+(c-1)(m+M_{j-1})$. Then on $V$ we have that 
\[\frac{\psi''_c}{c}<\frac{\psi_{j-1}-\frac{\delta}{2}}{1+\varepsilon_{j-1}}+\frac{(c-1)(m+M_{j-1})}{c}<\frac{\psi_{j-1}}{1+\varepsilon_{j-1}}-\frac{\delta}{4}+(c-1)(m+M_{j-1}).\]
Moreover, $\psi''_c>\varphi$ on $X$ and 
\begin{align*}
\int_{X'}(\psi''_c-\varphi)&=\int_{X'}(\psi'_c-\varphi)+(c-1)(m+M_{j-1})|X'|<(c-1)(1+m|X'|+M_{j-1}|X'|),\\
\int_{W}\psi''_c&=\int_{W}\psi'_c+(c-1)(m+M_{j-1})|W|<-j.
\end{align*}
If $\varepsilon_j>0$ is such that $\varepsilon_j<\varepsilon_{j-1}/2$, $\varepsilon_j(m+M_{j-1})<\delta/4$, $\varepsilon_j(1+m|X'|+M_{j-1}|X'|)<j^{-1}$, and $\psi_j:=\psi''_{1+\varepsilon_j}$, then $\varepsilon_j,\,\psi_j$ have the desired properties. 

We conclude that  $\varphi_j=(1+\varepsilon_j)^{-1}\psi_j$ is a decreasing sequence of negative smooth $\omega$-psh functions on $V$ such that $\varphi_j>(1+\varepsilon_j)^{-1}\varphi>\varphi$ on $X$. Hence
$\psi=\lim_{j\to\infty}\varphi_j$ is a negative $\omega$-psh function on $V$ and $\psi\geq\varphi$ on $X$. Note that 
\[\int_{X'}(\varphi_j-\varphi)=\frac{1}{1+\varepsilon_j}\int_{X'}(\psi_j-\varphi)-\frac{\varepsilon_j}{1+\varepsilon_j}\int_{X'}\varphi\to0\,,\,\;\int_{W}\varphi_j=\frac{1}{1+\varepsilon_j}\int_{W}\psi_j\to-\infty\,,\]
as $j\to+\infty$, hence $\psi=\varphi$ on $X$. 
\end{proof}

\section{Zariski open Stein covers}  \label{S:ex}

We study in this section the existence of adapted Zariski-open Stein covers on various types of K\"ahler manifolds. Since the extension property requires the existence of closed subvarieties of positive dimension, it is natural to explore in particular the case of projective varieties (many K\"ahler manifolds -e.g.\ generic tori or very generic K3 surfaces- do not admit 
any closed subvariety of positive dimension).

\subsection{Basic properties}\label{SS:BP}

We establish here the proof of Proposition \ref{P:cc}, and some of its consequences.

\smallskip

 $(i)$ Let $(V,\omega)$ be a compact K\"ahler manifold that satisfies condition (ZOS),
 and let $\omega'$ be another K\"ahler form cohomologous to $\omega$.
  By the $\ddbar$-Lemma, $\omega'=\omega+dd^c\eta$ for a smooth function $\eta:V\to\mathbb R$. Let $\{V_j\}_j$ be a Zariski-open Stein cover of $V$ such that $\omega=dd^c\rho_j$ on $V_j$, where $\rho_j$ is a smooth strictly psh exhaustion function on $V_j$. Then $\omega'=dd^c(\rho_j+\eta)$ and $\rho_j+\eta$ is a smooth strictly psh exhaustion function on $V_j$, since $\eta$ is globally bounded on $V$.
  Thus $(V,\omega')$ also satisfies condition (ZOS).
 
 \smallskip

$(ii)$ Let $\omega_1,\omega_2$ be K\"ahler forms on $V$ such that $(V,\omega_1),(V,\omega_2)$ verify condition (ZOS). If $\omega=a_1\omega_1+a_2\omega_2$, where $a_1,a_2>0$, then $(V,\omega)$ verifies condition (ZOS). Indeed, assume that $V=\bigcup_{j=1}^{N_1}V_j=\bigcup_{k=1}^{N_2}W_k$, where $V_j,W_k$ are open Stein sets such that $H_j=V\setminus V_j$, $L_k=V\setminus W_k$ are analytic subvarieties of $V$. Moreover, $\omega_1=dd^c\rho_j$ on $V_j$, $\omega_2=dd^c\eta_k$ on $W_k$, where $\rho_j,\eta_k$ are smooth strictly psh exhaustion functions on $V_j$, respectively on $W_k$. Then $V\setminus(V_j\cap W_k)=H_j\cup L_k$ is a proper analytic subvariety of $V$. Moreover, the function $\tau_{jk}:=a_1\rho_j+a_2\eta_k$ is a smooth strictly psh exhaustion function on $V_j\cap W_k$, hence $V_j\cap W_k$ is Stein. It follows that $\{V_j\cap W_k\}_{j,k}$ is an adapted Zariski-open Stein cover for $(V,\omega)$.
 
\smallskip

$(iii)$ If $(V,\omega)$ is a compact K\"ahler manifold which verifies condition (ZOS) and if $W$ is a (closed) complex submanifold of $V$ then it is easy to see that $(W,\omega|_W)$ verifies condition (ZOS), by intersecting the open sets $V_j$ with $W$.

\smallskip

$(iv)$ If $V$ is a projective manifold and $\omega$ is a Hodge form on $V$ then $(V,\omega)$ 
verifies condition (ZOS). Indeed, replacing $\omega$ by $k \omega$, we can assume that $V$ is an algebraic submanifold of the complex projective space ${\mathbb P}^N$, and $\omega=\omega_{FS}|_V$ is the Fubini-Study K\"ahler form. 
Using the canonical covering by affine charts ${\mathbb C}^N\hookrightarrow{\mathbb P}^N$, 
and the Fubini-Study potentials $\rho(z)=\log\sqrt{1+\|z\|^2}$, $z\in{\mathbb C}^N$, 
we have that $({\mathbb P}^N,\omega_{FS})$ verifies condition (ZOS), hence so does $(V,\omega)$ by 
the previous restriction property.

\smallskip

$(v)$ If $V$ is a projective manifold and $\{\omega\}=\sum_{\ell=1}^ma_\ell  \{ \omega_\ell \}$, where $a_\ell>0$ and 
$\{\omega_\ell\}$ are Hodge classes on $V$, then $(V,\omega)$ verifies condition (ZOS), as follows from the previous observations. It is easy to see that any K\"ahler class $\{\omega\}$ in the real Neron-Severi space $NS_{\R}(V)$ can be expressed as a convex combination of Hodge classes. For the convenience of the reader we include a proof in Lemma \ref{L:NS} below. Hence $(V,\omega)$ satisfies condition (ZOS) if the K\"ahler class $\{\omega\}\in NS_{\R}(V)$.

\begin{Lemma}\label{L:NS}
Let $(V,\omega)$ be a compact K\"ahler manifold such that $\{\omega\}\in NS_{\R}(V)$. Then $V$ is projective and $\{\omega\}=\sum_{\ell=1}^m a_\ell\vartheta_\ell\,$, where $a_\ell>0$ and $\vartheta_\ell$ are Hodge classes on $V$.
\end{Lemma}

\begin{proof} Let $\mathcal K_V\subset H^{1,1}(V,\R)$ denote the set of K\"ahler classes on $V$. Then $\mathcal K_V$ is a convex cone and it is open. Let $c_1(L_j)$, $L_j\in\Pic(V)$, $1\leq j\leq k$, be a basis of $NS_{\R}(V)$, so 
\[\{\omega\}=\sum_{j=1}^k c_jc_1(L_j)\,,\text{ where $c_j\in\R$.}\]
Since $\mathcal K_V$ is open it follows that there exists $\varepsilon>0$ such that the class 
\[\{\omega'\}=\sum_{j=1}^k r_jc_1(L_j)\in\mathcal K_V\,, \text{ if $r_j\in\Q$, $|r_j-c_j|<\varepsilon$.}\]
Note that $\{\omega'\}=N^{-1}\vartheta$, where $N\in\N$ and $\vartheta$ is a Hodge class on $V$, hence $V$ is projective. Since $\mathcal K_V\cap NS_{\R}(V)\neq\emptyset$ is open in $NS_{\R}(V)$, we have that $NS_{\R}(V)$ has a basis of K\"ahler classes, and hence a basis of Hodge classes, by the preceding discussion. So we can write 
\[\{\omega\}=\sum_{j=1}^k c_j\vartheta_j\,,\text{ where $\vartheta_j$ are Hodge classes and $c_j\in\R$.}\]
Choosing $r_j\in\Q$, $r_j<c_j$ sufficiently close to $c_j$, we infer by above that 
\[\{\omega\}=\sum_{j=1}^k r_j\vartheta_j+\sum_{j=1}^k a_j\vartheta_j=N^{-1}\vartheta+\sum_{j=1}^k a_j\vartheta_j\,,\]
where $a_j=c_j-r_j>0$, $N\in\N$ and $\vartheta$ is a Hodge class on $V$.
\end{proof}

\smallskip

$(vi)$ Let $(V^1,\omega_1),(V^2,\omega_2)$ be compact K\"ahler manifolds which verify condition (ZOS). Then $(V^1\times V^2,\omega)$ verifies condition (ZOS), where $\omega=\pi_1^\star\omega_1+\pi_2^\star\omega_2$ and 
$$
\pi_j:V^1\times V^2\to V^j
$$
 are the canonical projections.  Indeed, if $V^1=\bigcup_kV^1_k$ and $V^2=\bigcup_\ell V^2_\ell$ are adapted Zariski-open Stein covers for $(V^1,\omega_1)$ and $(V^2,\omega_2)$, then it is easy to see that $\{V^1_k\times V^2_\ell\}_{k,\ell}$ is an adapted Zariski-open Stein cover for $(V^1\times V^2,\omega)$.

\smallskip

$(vii)$ Let $V,W$ be compact K\"ahler manifolds and $f:W\to V$ be a (unbranched) holomorphic covering map. Assume that $\omega$ is a K\"ahler form on $V$ such that $(V,\omega)$ verifies condition (ZOS). Then $(W,f^\star\omega)$ verifies condition (ZOS). Indeed, since $W$ is compact $f$ is a finite map. If $U\subset V$ is an open Stein set such that $H=V\setminus U$ is an analytic subvariety of $V$ and if $\rho$ is a smooth strictly psh exhaustion function on $U$, then $\rho\circ f$ is a smooth strictly psh exhaustion function on $f^{-1}(U)$. Moreover $f^{-1}(H)$ is a proper analytic subvariety of $W$. 

\smallskip

In the case of a {\em branched} holomorphic cover $f:W\to V$, and with the above notation, it still holds that $f^{-1}(U)$ is a Zariski-open Stein set (see e.g.\ \cite[p.\ 49]{For17}). However in this case the cohomology class $\{f^\star\omega\}$ is merely semi-positive.
As noted by Matsumura \cite{Mat13}, strict positivity is required to expect the extension property to hold true
(see Proposition \ref{pro:matsumura}).

\subsection{The real Neron-Severi space}\label{SS:NS} 

A Fano manifold is a projective $n$-dimensional manifold $V$ whose anticanonical bundle
$-K_V$ is ample. It follows from the Kodaira-Nakano vanishing theorem that 
$H^{0,q}(V,\C)=0$ for $q>0$, so Hodge symmetry and Hodge decomposition theorem ensure that $H^2(V,\R)=H^{1,1}(V,\R)$. Hence
$$
NS_{\R}(V)=H^{1,1}(V,\R),
$$
and any K\"ahler class belongs to the real Neron-Severi space. Note also that in this case the map $c_1:H^1(V,\cO^\star_V)\cong\Pic(V)\to H^2(V,\Z)$ is an isomorphism. Thus condition (ZOS) is satisfied and the extension property holds for {\it any} K\"ahler class in this case.

Since $h^{0,2}(V)=\dim H^{0,2}(V,\C)$ is a birational invariant (see e.g.\ \cite{H77}), the extension property holds in particular on any {\it rational} manifold $V$ (i.e.\ birationally equivalent to $\P^n$), and hence on any toric manifold.

The Picard number $\rho(V)$ (rank of the Neron-Severi group) is always bounded from above 
by the Hodge number $h^{1,1}(V)$. 
Besides Fano manifolds, there are many manifolds
with maximal Picard number $\rho(V)=h^{1,1}(V)$, especially when $\dim_{\C} V \geq 3$.
We refer the interested reader to \cite{Be14}  for a list of examples and a classification
in complex dimension $2$. The extension property holds true for any K\"ahler class on such manifolds.

\smallskip

At the other extreme the Picard number is one  on a generic abelian variety, as well as on any generic projective $K3$ surface, so most K\"ahler classes do not belong
to the real Neron-Severi space and we do not know if the extension property holds for these.

\subsection{Projective toric manifolds}\label{SS:toric} 

Let $V$ be a compact toric manifold of dimension $n$. Then $V$ is a compactification of the complex torus $(\C^\star)^n$ such that the canonical action by multiplication of $(\C^\star)^n$ on itself extends to a holomorphic action of $(\C^\star)^n$ on $V$. As noted in Section \ref{SS:NS} the extension property is valid for any K\"ahler class on $V$, since $V$ is rational.

A toric K\"ahler form on $V$ is a K\"ahler form $\omega$ which is $(S^1)^n$-invariant. We give here an alternative proof that $(V,\omega)$ satisfies condition (ZOS). For this, we will need the following lemma:

\begin{Lemma}\label{L:toric} Let $f:[0,+\infty)^n\to\mathbb R$ be such that $\rho(z_1,\ldots,z_n):=f(|z_1|,\ldots,|z_n|)$ is a smooth strictly psh function on ${\mathbb C}^n$. Then $\rho$ is an exhaustion function for ${\mathbb C}^n$.
\end{Lemma}

\begin{proof} It follows from the maximum principle applied to $\rho$ on polydiscs that for each $1\leq j\leq n$ and for fixed $x_k\geq0$, $k\neq j$, the function $x_j\to f(x_1,\ldots,x_j,\ldots,x_n)$ is increasing. For $x_j>0$ and $x_k\geq0$, $k\neq j$, we obtain using Green's formula that 
\begin{equation}\label{e:av}
\begin{split}
\frac{\partial}{\partial x_j}\,f(x_1,\ldots,x_n)&=\frac{1}{2\pi}\,\frac{\partial}{\partial x_j}\,\int_0^{2\pi}\rho(x_1,\ldots,x_je^{it},\ldots,x_n)\,dt \\
&=\frac{2}{\pi x_j}\int_{\{|z_j|\leq x_j\}}\frac{\partial^2\rho}{\partial z_j\partial\overline z_j}\,(x_1,\ldots,z_j,\ldots,x_n)\,d\lambda(z_j),
\end{split}
\end{equation}
where $\lambda$ is the Lebesgue measure on $\mathbb C$. Let 
\[m:=\min\Big\{\frac{\partial^2\rho}{\partial z_j\partial\overline z_j}\,(0,\ldots,z_j,\ldots,0):\,|z_j|\leq1,\,1\leq j\leq n\Big\}.\]
We have $m>0$ since $\rho$ is strictly psh. Equation \eqref{e:av} implies that 
\[\frac{\partial}{\partial x_j}\,f(0,\ldots,x_j,\ldots,0)\geq\frac{2m}{x_j}\]
holds for $x_j\geq1$. Since $f$ is increasing in each variable we infer that 
\[f(x_1,\ldots,x_j,\ldots,x_n)\geq f(0,\ldots,x_j,\ldots,0)\geq2m\log x_j+f(0),\]
for every $1\leq j\leq n$, $x_j\geq 1$, and $x_k\geq 0$ for $k\neq j$. Hence 
\[\rho(z)\geq 2m\log^+\max\{|z_1|,\ldots,|z_n|\}+\rho(0),\;\forall\,z\in{\mathbb C}^n.\]
\end{proof}

\begin{Proposition}\label{P:toric} Let $V$ be a projective toric manifold of dimension $n$ equipped with a toric K\"ahler form $\omega$. Then $(V,\omega)$ verifies condition (ZOS).
\end{Proposition}

\begin{proof} Let $T:=(\C^\star)^n$ be the complex torus in $V$, and let $p_1,\ldots,p_N$ be the toric points of $V$. By \cite[Proposition 4.4]{ALZ} we have that $V=\bigcup_{j=1}^NV_j$, where the sets $V_j\ni p_j$ are open, $T\subset V_j$, $V\setminus V_j$ are analytic subvarieties of $V$, and the following properties hold: For $1\leq j\leq N$, there exists a biholomorphic map $\Phi_j:V_j\to\C^n$ such that $\Phi_j(p_j)=0$, $\Phi_j(T)=(\C^\star)^n\subset{\mathbb C}^n$, $V_j$ is invariant under the action of $T$ and the action of $T$ on $V_j$ coincides via $\Phi_j$ with the canonical action by multiplication of $(\C^\star)^n$ on ${\mathbb C}^n$. 

Set $\omega_j=(\Phi_j^{-1})^\star(\omega|_{V_j})$. Then $\omega_j$ is a K\"ahler form on ${\mathbb C}^n$, invariant under the canonical action of $(S^1)^n$ on ${\mathbb C}^n$. We have $\omega_j=dd^cv_j$ for some smooth strictly psh function $v_j$ on ${\mathbb C}^n$. It follows that 
\[w_j(z_1,\ldots,z_n):=\frac{1}{(2\pi)^n}\,\int_{[0.2\pi]^n}v_j(e^{i\theta_1}z_1,\ldots,e^{i\theta_n}z_n)\,d\theta_1...d\theta_n\]
is a smooth, polyradial, strictly psh function such that $dd^cw_j=\omega_j$. Lemma \ref{L:toric} shows that $w_j$ is an exhaustion function for ${\mathbb C}^n$. Thus $\rho_j:=w_j\circ\Phi_j$ is a smooth strictly psh exhaustion function on $V_j$ with $dd^c\rho_j=\omega$.
\end{proof}

For a brief description of the construction of $n$-dimensional projective toric manifolds $V$, and of very ample toric line bundles on $V$, from fans of cones in ${\mathbb R}^n$ we refer to \cite[Section 4]{ALZ} (see also \cite{Ful}). 

By Lemma \ref{L:toric} (and its proof) it follows that every smooth strictly psh function $\rho$ on $\C^n$ which is polyradial is an exhaustion, and in fact it grows at least like $c\log\|z\|$ for some constant $c>0$. We conclude this section by noting that the hypothesis of Lemma \ref{L:toric} that $\rho$ is polyradial is necessary. Namely, we give below an example of a smooth strictly psh function on $\C^n$ which is not an exhaustion, and an example of a smooth strictly psh exhaustion function on $\C^n$ which does not grow at least logarithmically (i.e.\ it is not bounded below by $c\log\|z\|$, for any constant $c>0$).

\begin{Example} Let $w_j=(j,0,\ldots,0)\in\mathbb C^n$, $j\geq1$, and $S=\{w_1,w_2,\ldots\}$. The function 
\[\displaystyle u(z)=\sum_{j=1}^\infty\frac{1}{2^{j+1}}\log(3\|z-w_j\|)\]
is psh on $\C^n$ and $\cC^\infty$-smooth on $\C^n\setminus S$. Moreover $u(z)\leq\frac{1}{2}\log^+\|z\|+C$ on $\C^n$ for some constant $C$, and $u(z)\geq0$ on $D:=\C^n\setminus\big(\bigcup_{j\geq1} B_j)$, where $B_j$ is the open ball in $\C^n$ centered at $w_j$ and of radius $\frac{1}{3}$. Let $\{a_j\}_{j\geq1}$ be a sequence such that $a_j<0$. By a regularized maximum construction on each $B_j$ we obtain a $\cC^\infty$-smooth psh function $v$ on $\C^n$ which has the following properties:

(i) $v=u$ on $D$ and $v(z)\leq\frac{1}{2}\log^+\|z\|+C$ on $\C^n$ for some constant $C$.

(ii) $v\geq a_j$ on $B_j$ and $v(w_j)=a_j$. 

Let $\rho(z)=\frac{1}{4}\log(1+\|z\|^2)+v(z)$. Then $\rho$ is a smooth strictly psh function $\C^n$ and by (i), 
\[\rho(z)=\frac{1}{4}\log(1+\|z\|^2)+u(z)\geq\frac{1}{4}\log(1+\|z\|^2) \text{ on $D$.}\]
Moreover $\rho$ is in the Lelong class, i.e.\ $\rho(z)\leq\log^+\|z\|+C$ on $\C^n$ for some constant $C$. Choosing $a_j=-\frac{1}{4}\log(1+j^2)$ we have $\rho(w_j)=0$ so $\rho$ is not an exhaustion. If we take $a_j=\log\log(j+1)-\frac{1}{4}\log(1+j^2)$ we obtain a smooth strictly psh exhaustion function $\rho$ such that $\rho(w_j)=\log\log(j+1)$. 
\end{Example}

\subsection{Positivity vs.\ extension property}\label{SS:Matsumura}

We finally study a necessary condition for a cohomology class $\alpha$ to have the extension property.
Recall that a class $\alpha \in H^{1,1}_{psef}(V,\R)$ is pseudo-effective if it is the deRham cohomology class of a positive closed current $T$ of bidegree $(1,1)$. Fixing $\theta \in \alpha$ a smooth closed $(1,1)$-form, it follows from the $\partial\overline{\partial}$-lemma that $T=\theta+i\partial\overline{\partial} \varphi$, for some function $\varphi \in PSH(V,\theta)$. Recall that a class $\alpha \in H^{1,1}(V,\R)$ is numerically effective (nef) if it lies in the closure of the K\"ahler cone of $V$. 

\begin{Definition}\label{D:ext}
We say that a class $\alpha=\{\theta\} \in H^{1,1}_{psef}(V,\R)$ has the {\em extension property} if $PSH(V,\theta)|_X=PSH \left(X,\theta|_X\right)$ holds for any irreducible analytic subset $X \subset V$. We say that $\alpha$ has the {\em bounded extension property} if, for any irreducible analytic subset $X\subset V$, any $\theta|_X$-psh function on $X$ that is bounded below near some point $p\in X$ has a $\theta$-psh extension to $V$ that is bounded below near $p$ in $V$.
\end{Definition}

Inspired by Matsumura \cite{Mat13} we prove the following result:

\begin{Proposition} \label{pro:matsumura}
Let $V$ be a projective manifold of dimension $n\geq2$ and let $\alpha \in H^{1,1}_{psef}(V,\R)$, $\alpha\neq0$, be a pseudo-effective class.

(i) If $\alpha$ has the extension property then $\alpha$ is nef.

(ii) If $\alpha$ has the bounded extension property then $\alpha$ is a K\"ahler class.
\end{Proposition}


This result is due to  Matsumura \cite[Theorem 1.2]{Mat13} when $\alpha=c_1(L)$ is the first Chern class of some holomorphic line bundle 
$L$ on $V$. His strategy of proof can be adapted replacing the algebraic Nakai-Moishezon criterion by the transcendental version due to Demailly-Paun \cite{DP04}. A similar observation has been made by Meng-Wang in \cite{MW21}, who moreover noticed that a stronger extension property is needed for $\alpha$ to be a K\"ahler class. We give here a direct proof of Proposition \ref{pro:matsumura}, without using the results of \cite{DP04}.

\begin{proof}[Proof of Proposition \ref{pro:matsumura}]
We fix a K\"ahler form $\omega_V$ on $V$. Let $\theta \in \alpha$ be a smooth closed $(1,1)$-form and let $\varphi_0 \in PSH(V,\theta)$.

We observe first that the restriction of $\alpha$ to any irreducible curve ${\mathcal C}\subset V$ is K\"ahler, provided that ${\mathcal C}=D_1 \cap \cdots \cap D_{n-1}$ is a complete intersection of ample divisors (this is in analogy to \cite[Lemma 2.1]{Mat13}). Indeed, let $h_j$ be a Hermitian metric on the line bundle $\cO_V(D_j)$ such that $\omega_j:=c_1(\cO_V(D_j),h_j)$ is a K\"ahler form. By the Lelong-Poincar\'e formula we have $[D_j]=\omega_j+dd^c\log|S_j|_{h_j}$, where $[D_j]$ is the current of integration on $D_j$ and $S_j$ is the canonical section of $\cO_V(D_j)$. Fix $\varepsilon>0$ such that $\omega_j\geq\varepsilon\omega_V$ for $1\leq j\leq n-1$. Since $\mathcal C$ is a complete intersection it follows from \cite{De93} that $[C]=[D_1]\wedge\ldots\wedge[D_{n-1}]$. Therefore
\[\int_{\mathcal C}\theta=\int_V\theta\wedge[\mathcal C]=\int_V(\theta+dd^c\varphi_0)\wedge\omega_1\wedge\ldots\wedge\omega_{n-1}\geq\varepsilon^{n-1}\int_V(\theta+dd^c\varphi_0)\wedge\omega_V^{n-1}>0,\]
as $\alpha\neq0$. Hence $\alpha|_{\mathcal C}$ is K\"ahler. 

Next, a Bertini type theorem \cite[Theorem 1.3]{Zha09} ensures that for any $p,q \in V$ there exists a smooth complete intersection curve ${\mathcal C}$ as above with $p,q \in {\mathcal C}$ (see \cite[Theorem 2.2]{Mat13}). 

\medskip

$(i)$ Following Demailly we consider
\[\Psi_\theta(x):=\left( \sup \{ \varphi(x):\, \varphi \in PSH(V,\theta) \text{ with } \varphi \leq 0 \} \right)^*.\]
This is a $\theta$-psh function with minimal singularities on $V$. Fix $x \in V$ and ${\mathcal C}$ a smooth curve as above with $x \in {\mathcal C}$. Since $\alpha|_{\mathcal C}$ is K\"ahler, one can find $\phi \in PSH({\mathcal C}, \theta|_{\mathcal C})$ that is bounded below on ${\mathcal C}$. It follows from the extension property that there exists $\varphi \in PSH(V,\theta)$ such that $\varphi|_{\mathcal C}=\phi$. Shifting $\varphi,\phi$ by a constant, we can assume that $\varphi \leq 0$, hence $\Psi_\theta(x) \geq \varphi(x)=\phi(x)>-\infty$. Thus $\Psi_\theta$ has zero Lelong number at each $x\in V$, so by Demailly's regularization theorem \cite{De92} one can find smooth functions  $v_{\varepsilon} \in PSH(V,\theta+\varepsilon \omega_V)$ decreasing to $\Psi_\theta$. Therefore $\alpha$ is a nef class, limit of the K\"ahler classes $\alpha+2\varepsilon \{ \omega_V \}$.

\medskip

$(ii)$ Using the bounded extension property, we start by constructing, for every point $p\in V$, a function $\psi_p\in PSH(V,\theta)\cap L^\infty(V)$ such that 
\[\theta+dd^c\psi_p\geq\omega_V \text{ in a neighborhood $W_p$ of $p$.}\]

Let $B(p,r)$ be a small coordinate ball centered at $p=0$. Let $A>0$ such that $Add^c\|z\|^2 +\theta \geq \omega_V$ in $B(p,r)$. For $q\in V\setminus B(p,r)$ we let ${\mathcal C}$ be a smooth curve as above passing through $p$ and $q$. As $\alpha|_{\mathcal C}$ is K\"ahler, we can find a $\theta$-psh function $u$ on ${\mathcal C}$ such that $u(p)=-\infty$ and $u$ is bounded below near $q$ (see e.g.\ \cite{CG09}). By the bounded extension property, $u$ is the restriction of a $\theta$-psh function $U_q$ on $V$ that is bounded below on a neighborhood $G_q$ of $q$. Since $V\setminus B(p,r)$ is compact, we have $V\setminus B(p,r)\subset \bigcup_{j=1}^N G_{q_j}$, for some $q_j\in V\setminus B(p,r)$, $1\leq j\leq N$. Let $\phi=\max_{1\leq j\leq N}U_{q_j}+C$. Then $\phi\in PSH(V,\theta)$, $\phi>0$ on $V\setminus B(p,r)$ if $C$ is sufficiently large, and $\phi(p)=-\infty$. We now use a gluing construction: the function
\[\psi_p=\begin{cases}\phi  & \text{ in } V \setminus B(p,r),\\\max\left\{A (\|z\|^2-r^2), \phi\right\} & \text{ in } B(p,r),\end{cases}\]
is bounded and $\theta$-psh on $V$, and it coincides with $A (\|z\|^2-r^2)$ in some ball $B(p,\delta)\subset B(p,r)$. Hence $\theta+dd^c\psi_p\geq\omega_V$ on $W_p=B(p,\delta)$.

We now pick a finite cover $V=\bigcup_{j=1}^\ell W_{p_j}$, and we set 
\[\psi=\frac{1}{\ell}\,\sum_{j=1}^\ell \psi_{p_j}.\]
Then $\psi\in PSH(V,\theta)\cap L^\infty(V)$ and $\theta+dd^c\psi_p\geq\frac{1}{\ell}\,\omega_V$ on $V$. By Demailly's regularization theorem \cite{De92} there exists a smooth $\theta$-psh function $\varphi\geq\psi$ such that  $\theta+dd^c\varphi\geq\frac{1}{2\ell}\,\omega_V$, hence $\alpha$ is a K\"ahler class.
\end{proof}

 \end{document}